\documentclass[twoside]{amsart}
\usepackage{mathrsfs}
\usepackage{amsfonts}
\usepackage{amssymb}
\usepackage{amssymb}
\usepackage{amssymb}
\usepackage{amssymb}
\usepackage{amssymb}
\usepackage{amssymb}
\usepackage{amssymb}
\usepackage{amssymb}
\usepackage{amsmath}
\usepackage{mathdots}
\usepackage[all]{xy}

\usepackage{lastpage}

\RequirePackage{amsmath} \RequirePackage{amssymb}
\usepackage{amscd,latexsym,amsthm,amsfonts,amssymb,amsmath,amsxtra}
\usepackage[colorlinks=true, urlcolor=blue, linkcolor=red,anchorcolor=blue,citecolor=blue]{hyper ref}

    \newcommand{\CO}{{\mathcal {O}}}

    \newcommand{\CW}{{\mathcal {W}}}

    \newcommand{\lenth}{{\mathrm {\lenth}}}

    \theoremstyle{plain}

    \newtheorem{thm}{Theorem}[section] \newtheorem{corollary}[thm]{Corollary}
    \newtheorem{lemma}[thm]{Lemma}  \newtheorem{proposition}[thm]{Proposition}
     \newtheorem{definition}[thm]{Definition}

    \numberwithin{equation}{section}

\usepackage[top=1in, bottom=1in, left=1.25in, right=1.25in]{geometry}

\begin{document}
\author{Jingsong Chai}
\title{A weak kernel formula for Bessel functions}
%\date{Aug, 2014}
\maketitle{}

\begin{abstract}
In this paper, we prove a weak kernel formula of Bessel functions
attached to irreducible generic representations of p-adic $GL(n)$.
As an application, we show that the Bessel function defined by
Bessel distribution coincides with the Bessel function defined via
uniqueness of Whittaker models on the open Bruhat cell.
\end{abstract}

\maketitle

\section{Introduction}

For a generic irreducible smooth representation $\pi$ of $GL(n,F)$
with its contragredient $\tilde{\pi}$, where $F$ is a p-adic field,
there are two ways to attach Bessel functions to $\pi$. The first is
via Bessel distribution $B_{l,l'}$, where $l,l'$ are Whittaker
functionals on $\pi$ and $\tilde{\pi}$ respectively. Such
distributions were used by Gelfand and Kazhdan
(\cite{GelfandKazhdan:1975}), and Shalika (\cite{Shalika:1974}) to
prove uniqueness of Whittaker functionals. In \cite{Baruch:2001},
for more general quasi-split groups, E.Baruch showed that the
restriction of $B_{l,l'}$ to the big open Bruhat cell is given by a
locally constant function $j_0(g)$. \\

On the other hand, for Whitaker function $W\in \mathcal{W}(\pi)$,
where $\mathcal{W}(\pi)$ denotes the Whittaker model of $\pi$, the
integral

$$\int_{N_n} W(gu)\psi^{-1}(u)du$$
converges in the stable sense if $g$ is in the big open cell, and
thus defines a function $j_{\pi}(g)$ there such that

$$j_\pi(g)W(I)=\int_{N_n} W(gu)\psi^{-1}(u)du$$
because of uniqueness of Whittaker functionals. This Bessel function
$j_{\pi}$ was first defined in this way by David Soudry in
\cite{Soudry:1984} for $GL(2,F)$, and then was generalized by
E.Baruch to $GL(n,F)$. For more details see
\cite{Baruch:2003,Baruch:2005}. \\

In the case $GL(3,F)$ (also $GL(2,F)$), E.Baruch in
\cite{Baruch:1997,Baruch:2004} proved the Bessel function
$j_{\pi}(g)$ is locally integrable on the whole group, and gives the
Bessel distribution $B_{l,l'}$ on $GL(3,F)$, which implies that the
above two functions $j_0,j_{\pi}$ are the same. \\

These Bessel functions and Bessel distributions have many
applications to the theory of automorphic forms, to list a few, for
example see \cite{BaruchMao:2003, BaruchMao:2005, BaruchMao:2007,
CKPSSH:2004, CPSSH:2008, Shahidi:2002}. Thus it is desirable to
generalize E.Baruch's important results to more general $GL(n,F)$.
The obstacle is the local integrability of $j_{\pi}(g)$. This is
done by E.Baruch in \cite{Baruch:2004} for $GL(3,F)$ case
using Shalika germs, but seems to be very difficult in general. \\

In this paper, we generalize some above results to $GL(n,F)$. More
precisely, we proved the
following result. \\

\let\thefootnote\relax\footnote{
\textit{Mathematics Subject Classification (2010).} Primary: 22E50; Secondary:11F70. \\

\textit{Keywords:} Bessel functions; Bessel distributions; weak
kernel formula. \\}

%\clearpage

\textbf{Theorem 1}. If $\pi$ is irreducible, smooth and generic,
then we have

\[
j_0(g)=j_{\pi}(g)
\]
for all $g\in N_n\omega_nA_nN_n$, where $N_n$ is the upper
triangular unipotent subgroup, $A_n$ is the subgroup of diagonal
matrices, and $\omega_n$ is the longest element in the Weyl group.
\\

Once we know local integrability of Bessel functions, we can also
show $B_{l,l'}$ is given by $j_{\pi}$ on the whole group following
the method in \cite{Baruch:2004}. \\

We here essentially follow Baruch's approach, and an important
ingredient in the proof is a kernel formula, which has its own
interests and can be stated as
follows. \\

\textbf{Theorem 2}(Theorem \ref{thm1}, Theorem \ref{thm3}). Assume
either $\pi$ is supercuspidal and $W$ is any Whittaker function of
$\pi$, or $\pi$ is irreducible smooth and generic and $W$ is the
normalized Howe vectors with sufficiently large level(see section 5
for the definition of Howe vectors). For any $b\omega_n$,
$b=diag(b_1,...,b_n)\in A_n$, and any $W\in \mathcal{W}(\pi,\psi)$,
we have

$$W(b\omega_n)=$$
$$\int j_{\pi}\left(b\omega_n \begin{pmatrix} a_1 & \\ x_{21} & a_2 \\ & & \ddots \\
x_{n-1,1} & \cdots & x_{n-1,n-2} & a_{n-1} \\ & & & &  1
\end{pmatrix}^{-1} \right) W\begin{pmatrix} a_1 & \\ x_{21} & a_2 \\
& & \ddots \\ x_{n-1,1} & \cdots & x_{n-1,n-2} & a_{n-1}\\ & & & & 1
\end{pmatrix} $$
$$|a_1|^{-(n-1)}da_1|a_2|^{-(n-2)}dx_{21}da_2\cdots |a_{n-1}|^{-1}dx_{n-1,1}\cdots dx_{n-1,n-2}da_{n-1} $$
where the right side is an iterated integral, $a_i$ is integrated
over $F^{\times}\subset F$ for $i=1,...,n-1$, $x_{ij}$ is integrated
over $F$ for all relevant $i,j$, and all measures are additive
self-dual Haar measures on $F$. \\

We remark that in the case of general generic representation $\pi$,
this kernel formula is expected to be true for a wide class of
Whittaker functions of $\pi$(though not all of them), but currently
we are only able to prove it for Howe vectors which is sufficient
for the purpose of this paper. \\

Such formula was first proved by David Soudry in \cite{Soudry:1984}
for generic irreducible representations of $GL(2,F)$, and then was
generalized to $GL(3,F)$ by E.Baruch in \cite{Baruch:2004}. Due to
the lack of local integrability of $j_{\pi}$, we have to write the
above
integral as an iterated integral. \\

E.Baruch in \cite{Baruch:2001} showed the existence of $j_0$ for
smooth generic irreducible representations of quasi-split reductive
groups over local fields of characteristic zero. Recently E.Lapid
and Zhengyu Mao in \cite{LapidMao:2013} defined $j_{\pi}$ using
uniqueness of Whittaker functionals for split reductive groups. It
is interesting to see if the results here can be generalized to
these cases. \\

The paper is organized as follows. In section 2 and 3, we recall
some results about Bessel functions $j_{\pi}(g)$ and Bessel
distributions $B_{l,l'}$. Section 4 is to prove the weak kernel
formula. Section 5 is devoted to prove some properties of Howe
vectors, which will be needed later. In section 6 we show that these
two Bessel functions are equivalent in the supercuspidal case. In
the last section we generalize the results to generic case. \\

\textbf{Notations.}\\

Let $F$ be a p-adic field with ring of integers $\mathcal{O}$, use
$|\cdot|$ to denote the valuation on $F$. We will always fix a
self-dual Haar measure on $F$. Let $K_n=GL(n,\mathcal{O})$ and
$G_n=GL(n,F)$, $\pi$ a generic irreducible smooth admissible
representation of $G_n$, with its contragredient $\tilde{\pi}$. Let
$N_n$ be the maximal unipotent subgroup of upper triangular
matrices. Let $A_n$ be the group of diagonal matrices. Let
$\bar{N_n}$ be the transpose of $N_n$. $B_n=A_nN_n$,
$\bar{B_n}=A_n\bar{N_n}$. \\

Let $\psi$ be a nontrivial additive character of $F$ with conductor
exactly $\mathcal{O}$.  We extend $\psi$ to a character of $N_n$ by
$\psi(u)=\psi(\sum_{i=1}^{n-1} u_{i,i+1} )$ if $u=(u_{ij})\in N_n$,
and still denote it as $\psi$. \\

Use $l,l'$ to denote the Whittaker functionals on $\pi,\tilde{\pi}$
with respect to $\psi$ and $\psi^{-1}$, respectively. Let
$\mathcal{W}=\mathcal{W}(\pi,\psi)$,
$\mathcal{\widetilde{W}}=\mathcal{\widetilde{W}}(\tilde{\pi},\psi^{-1})$
be the corresponding Whittaker models.  \\

Let $\mathbb{W}$ be the Weyl group of $G_n$, and use $\omega_n$ to
denote the longest Weyl element in $\mathbb{W}$, i.e.
$\omega_n=\begin{pmatrix} 0 & & 1 \\ &  \iddots \\ 1 & & 0
\end{pmatrix}$.     \\

\section{Bessel functions}

In this section, we review some results about Bessel functions
$j_{\pi}(g)$ in \cite{Baruch:2005}. For our purpose, we will
restrict to supercuspidal representations, though most of the
notions and results can be generalized to smooth irreducible generic
representations. We refer to \cite{Baruch:2005} for more
details in general. \\

So let $(\pi,V)$ be an irreducible supercuspidal representation of
$G_n$. If $\psi$ is a nondegenerate character of $N_n$, use
$\mathcal{W}=\mathcal{W}(\pi,\psi)$ to denote the Whittaker model of
$\pi$ w.r.t. $\psi$. If $M>0$ is a positive constant, let

\[
A_n^M=\{a\in A_n: |\frac{a_i}{a_{i+1}}|<M \ \ for \ \ i=1,2,...,n-1,
if \ \ a=diag(a_1,...,a_n) \}
\]
Note that as $M\to \infty$, $A_n^M$ cover $A_n$. \\

We start with the following important result of Baruch. \\

\begin{proposition}
\label{prop1} For any $W\in \mathcal{W}(\pi,\psi)$, $M>0$, the
function on $A_n^M\times N_n$ defined by

\[
(a,u)\to W(a\omega_n u)
\]
is compactly supported in $N_n$ with support independent of $a\in
A_n^M$. That is, if $W(a\omega_nu)\ne 0$, with $a\in A_n^M$, $u\in
N_n$, then there exists a compact subset $U\subset N_n$, which is
independent of $a$, such that $u\in U$.
\end{proposition}

\begin{proof}
This follows from Theorem 5.7 and Lemma 6.1 in \cite{Baruch:2005}.
\end{proof}

The above result allows us to define Bessel functions for
supercuspidal representations as follows. Take $W\in \mathcal{W}$.
Consider the integrals for $g\in N_nA_n\omega_n N_n$

\[
\int_{Y_i}W(gu)\psi^{-1}(u)du
\]
where $Y_1\subset Y_2\subset ... \subset Y_i\subset Y_{i+1}\subset
...$ is an increasing filtration of $N_n$ with compact open
subgroups.  \\

By Proposition \ref{prop1}, if $Y_i$ is large enough, these
integrals become stable. The stable limit is independent of the
choices of sequence $\{Y_i \}$. Use $\int_{N_n}^*$ to denote this
limit, which defines a nontrivial Whittaker functional on
$\mathcal{W}$. Thus there exists a scalar $j_{\pi,\psi}(g)$ such
that

\[
\int_{N_n}^* W(gu)\psi^{-1}(u)du=j_{\pi,\psi}(g)W(I)
\] \\

\begin{definition}
\label{def2} The assignment $g\to j_\pi(g)=j_{\pi,\psi}$ defines a
function on $N_nA_n\omega_n N_n$, which is called the Bessel
function of $\pi$ attached to $\omega_n$.
\end{definition}

We extend $j_{\pi}$ to $G_n$ by putting $j_{\pi}(g)=0$ if $g\notin
N_nA_n\omega_nN_n$, and still use $j_{\pi}$ to denote it and call it
the Bessel function of $\pi$. \\

$\bullet$ $j_{\pi}$ is locally constant on $N_nA_n\omega_n N_n$. \\

$\bullet$ For any $u_1, u_2\in N_n$, any $g\in G_n$, we have
$j_{\pi}(u_1gu_2)=\psi(u_1)\psi(u_2)j_{\pi}(g)$. \\

$\bullet$ One may also attach Bessel functions to other Weyl
elements. For more details, see \cite{Baruch:2005}. For Bessel
functions defined
in this way for split reductive groups, see \cite{LapidMao:2013}. \\

For $W\in \mathcal{W}$, let $\widetilde{W}(g)=W(\omega_n\cdot
{^t}g^{-1})$, then $\{\widetilde{W}:W\in \mathcal{W} \}$ is the
Whittaker model for the contragredient $\widetilde{\pi}$ with
respect to $\psi^{-1}$. By \textbf{Corollary 8.5} in
\cite{Baruch:2005}, we have the following relation

\begin{equation}
j_{\tilde{\pi},\psi^{-1}}(g)=j_{\pi,\psi}(g^{-1}), \ \ \ \ \ \ g\in
B_n\omega_nB_n \label{1}
\end{equation} \\

\section{Bessel distributions}

In this section, we collect some useful properties about Bessel
distributions. Let $\pi$ be an irreducible smooth generic
representation of $G_n$. Let $\pi^{*}$ and $\widetilde{\pi}^{*}$
denote the linear dual of $\pi$ and $\widetilde{\pi}$ respectively.
Let $f$ be a locally constant function with compact support on
$G_n$, take $l\in \pi^{*}$, $l'\in \widetilde{\pi}^{*}$. Define
$\widetilde{\pi}(f)l'$ as

$$\widetilde{\pi}(f)l'=\int_{G_n}f(g)\widetilde{\pi}(g)l'dg$$
or equivalently, for any $\tilde{v}\in \widetilde{\pi}$,

$$<\widetilde{\pi}(f)l', \tilde{v}>=\int_{G_n}f(g)<\widetilde{\pi}(g)l',\tilde{v}>dg$$
$$=\int_{G_n} f(g)<l', \widetilde{\pi}(g^{-1})(\tilde{v})>dg$$
then $\widetilde{\pi}(f)l'$ is a smooth linear functional on
$\widetilde{\pi}$, hence can be identified with a vector $v_{f,
l'}\in \pi$. \\

\begin{definition}
\label{def2} Define Bessel distribution $B_{l,l'}(f)$ as

$$B_{l,l'}(f)=l(v_{f,l'})$$
\end{definition}

$\bullet$ The definition of $B_{l,l'}$ depends on $l,l'$ and Haar measure $dg$ on $G_n$. \\

$\bullet$ When both $l,l'$ are Whittaker functionals, such
$B_{l,l'}$ is the Bessel distribution first studied by Gelfand and
Kazhdan (\cite{GelfandKazhdan:1975})  for $GL_n$ in p-adic case, and
by Shalika (\cite{Shalika:1974}) in archimedean case, by Baruch
(\cite{Baruch:2001}) for quasi-split groups in both
non-archemedean and archemedean cases.  \\

$\bullet$ It was shown in \cite{Baruch:2001} that when both $l,l'$
are Whittaker functionals with respect to $\psi$ and $\psi^{-1}$
respectively, $B_{l,l'}$ can be represented by a locally constant
function in non-archimedean case, and by a real analytic function in
archimedean case, when restricted to the open Bruhat cell. We will
denote this
function by $j_0(g)$, $g\in N_nA_n\omega_n N_n$. \\

$\bullet$ When $l'$ is the Whittaker functional on $\widetilde{\pi}$
with respect to $\psi^{-1}$, for any $\hat{W}_{\widetilde{v}}\in
\CW(\widetilde{\pi},\psi^{-1})$, we have

$$<v_{f,l'}, \widetilde{v}>=\int_{G_n} f(g)\hat{W}_{\widetilde{v}}(g^{-1})dg $$
\\

From now on, we will always assume $l,l'$ are nonzero Whittaker
functionals. Fix $l$ on $\pi$ with respect to $\psi$, for $W\in
\mathcal{W}(\pi,\psi)$, $\hat{W}\in
\mathcal{W}(\tilde{\pi},\psi^{-1})$, if either $W\begin{pmatrix} h &
\\ & 1 \end{pmatrix}$ or $\hat{W} \begin{pmatrix} h & \\ & 1
\end{pmatrix}$ is compactly supported mod $N_{n-1}$, by results in \cite{Bernstein:1984}, we can
normalize $l'$, so that

$$<v, \tilde{v}> =\int_{N_{n-1}\backslash G_{n-1}} W_v
\begin{pmatrix} h & \\ & 1 \end{pmatrix}\hat{W}_{\tilde{v}}
\begin{pmatrix} h & \\ & 1 \end{pmatrix}dh$$
where the right side integral defines a $P_n$ invariant pairing
between $\pi$ and $\tilde{\pi}$, here $P_n$ is the so-called
mirabolic subgroup of $G_n$. \\

\begin{lemma}
\label{le1} With the above normalization, if $\hat{W}_{\tilde{v}}
\begin{pmatrix} h & \\ & 1 \end{pmatrix}$ is compactly supported mod
$N_{n-1}$, we have for any $f\in C_c^{\infty}(G)$,

$$\int_{G_n}
f(g)\hat{W}_{\tilde{v}}(g^{-1})dg=\int_{N_{n-1}\backslash G_{n-1}}
B_{l,l'}\left( L\begin{pmatrix} h & \\ & 1 \end{pmatrix}.f
\right)\hat{W}_{\tilde{v}} \begin{pmatrix} h & \\ & 1
\end{pmatrix}dh$$
where $L$ denotes the left action of $G_n$ on $f$.
\end{lemma}

\begin{proof}
Let $v_{f,l'}$ be the vector in $\pi$ as in Definition \ref{def2},
so $B_{l,l'}(f)=l(v_{f,l'})$. For $\widetilde{v}\in
\widetilde{\pi}$, denote $\underline{h}=\begin{pmatrix} h & \\ & 1
\end{pmatrix}$,  we have

$$<v_{L(\underline{h})f, l'}, \widetilde{v}>=\int_{G_n}
f(\underline{h}^{-1}g)<l', \widetilde{\pi}(g^{-1})\widetilde{v}>dg$$

$$=\int_{G_n} f(g)<l', \widetilde{\pi}(g^{-1}) \widetilde{\pi}(\underline{h}^{-1})\widetilde{v}>dg$$

$$=<v_{f,l'},
\widetilde{\pi}(\underline{h}^{-1})\widetilde{v}>=<\pi(\underline{h})v_{f,l'},
\widetilde{v}>$$ Thus $B_{l,l'}\left(  L\begin{pmatrix} h & \\ & 1
\end{pmatrix}.f  \right)=l(v_{L(\underline{h})f,
l'})=l(\pi(\underline{h})v_{f,l'})=W_{v_{f,l'}}(\underline{h})$. \\

It follows that the right side in lemma is $<v_{f,l'},
\widetilde{v}>$ by normalization, which equals the left side.
\end{proof}

\section{Kernel formula}

In this section, we will prove a weak kernel formula for Bessel
functions attached to supercuspidal representations as in section 1,
which is the first main result of this paper. The method of the
proof follows that of Baruch in \cite{Baruch:2004} by generalizing
corresponding results there to $GL(n)$. \\

So through out of this section, $\pi$ will be an irreducible
supercuspidal representation of $G_n$. Let $Y_i$ be the unipotent
part of the parabolic subgroup of $G_n$ associated to the partition
{(n-i+1,1,...,1)}, $1\leq i\leq n$. Note that $Y_1=\{I_n\}$,
$Y_n=N_n$. By
Proposition \ref{prop1}, we have the following lemma. \\

\begin{lemma}
\label{le2} For any $b\omega_n$, $b=diag(b_1,...,b_n)\in A_n$,  then
as a function of $u_i\in F^{n-i}$, the function

$$\int_{Y_i} W\left( b\omega_ny_i \begin{pmatrix} I & u_i &  \\  & 1 & \\  &
& I_{i-1} \end{pmatrix}\right)\psi^{-1}(y_i)dy_i$$ is compactly
supported, where $\psi(y_i)$ is the restriction of Whittaker
character to $Y_i\subset N_n$.
\end{lemma}

\begin{proof}
By Proposition \ref{prop1}, if $b\in A_n^M$ for some constant $M>0$,
then the function $W(b\omega_n u)$ is compactly supported as a
function $u\in N_n$, with support independent of $b\in A_n^M$. Then
its restriction to

\[
\{y_i\begin{pmatrix} I & u_i &  \\  & 1 & \\  & & I_{i-1}
\end{pmatrix}: y_i\in Y_i, u_i\in F^{n-i}\}
\]
is again compactly supported. Now the lemma follows immediately.
\end{proof}

\begin{thm}
\label{thm1} (weak kernel formula) For any $b\omega_n$,
$b=diag(b_1,...,b_n)\in A_n$, and any $W\in \mathcal{W}$, we have

$$W(b\omega_n)=$$
$$\int j_{\pi}\left(b\omega_n \begin{pmatrix} a_1 & \\ x_{21} & a_2 \\ & & \ddots \\
x_{n-1,1} & \cdots & x_{n-1,n-2} & a_{n-1} \\ & & & &  1
\end{pmatrix}^{-1} \right) W\begin{pmatrix} a_1 & \\ x_{21} & a_2 \\
& & \ddots \\ x_{n-1,1} & \cdots & x_{n-1,n-2} & a_{n-1}\\ & & & & 1
\end{pmatrix} $$
$$|a_1|^{-(n-1)}da_1|a_2|^{-(n-2)}dx_{21}da_2\cdots |a_{n-1}|^{-1}dx_{n-1,1}\cdots dx_{n-1,n-2}da_{n-1} $$
where the right side is an iterated integral, $a_i$ is integrated
over $F^{\times}\subset F$ for $i=1,...,n-1$, $x_{ij}$ is integrated
over $F$ for all relevant $i,j$, and all measures are additive
self-dual Haar measures on $F$.
\end{thm}

\begin{proof}
The proof is based an inductive argument. We begin with the proof of
the following identity, for any $b\omega_n$,

$$\int_{Y_{n-1}} W\left( b\omega_ny_{n-1} \right)\psi(-y_{n-1})dy_{n-1} =$$

\begin{equation}
\int_{F^{\times}} j_{\pi}\left( b\omega_n\begin{pmatrix} a_1^{-1} &  \\
& I_{n-1}
\end{pmatrix}\right) W\begin{pmatrix} a_1 &  \\  & I_{n-1}
\end{pmatrix}|a_1|^{-(n-1)}da_1 \label{2}
\end{equation}

where $da_1$ is the additive Haar measure on $F$.  Note that because
of Proposition \ref{prop1}, the left side integral is absolutely
convergent. Because $\pi$ is supercuspidal, $W$ is compactly
supported mod $N_nZ_n$, where $Z_n$ is the center of $G_n$, then $W\begin{pmatrix} a_1 &  \\
& I_{n-1}
\end{pmatrix}$ is compactly supported in $F^\times$ as a function of $a_1$. Since $j_{\pi}$
is locally constant on the big cell, $j_{\pi}\left(
b\omega_n\begin{pmatrix} a_1^{-1} &  \\  & I_{n-1}
\end{pmatrix}\right)$ is also locally constant as a function of
$a_1$, then the right side integral reduces to a finite sum, and
hence is also absolutely convergent.  \\

For this consider the following function $M_{n-1}(x): F\to
\mathbb{C}$ by

\[
M_{n-1}(x)=\int_{Y_{n-1}} W\left( b\omega_ny_{n-1}\begin{pmatrix} 1 & x & \\
& 1 &
\\ &  & I_{n-2} \end{pmatrix}\right) \psi(-y_{n-1})dy_{n-1}
\]

By Lemma \ref{le2}, this is a compactly supported function in $x$.
Thus its Fourier transform $\widehat{M}_{n-1}(y)$ is also compactly
supported, and we have Fourier inversion formula

\[
M_{n-1}(x)=\int_{F} \widehat{M}_{n-1}(y)\psi(yx)dy=\int_{F^{\times}}
\widehat{M}_{n-1}(y)\psi(yx)dy
\]
where the last equality follows from the facts that $dy$ is the
additive Haar measure, and $F^{\times}$ is of full measure in $F$.
\\

Put $x=0$, we get

\[
M_{n-1}(0)=\int_{F^{\times}} \widehat{M}_{n-1}(y)dy
\]

Now we compute the Fourier coefficient $\widehat{M}_{n-1}(y)$ when
$y=a_1\neq 0$.

\begin{eqnarray*}
&& \widehat{M}_{n-1}(a_1) \\
&=&\int_F\int_{Y_{n-1}} W\left(
b\omega_ny_{n-1}\begin{pmatrix} 1 & x & \\  & 1 & \\ &  & I_{n-2}
\end{pmatrix}\right)\psi(-y_{n-1})\psi(-a_1x)dy_{n-1}dx
\\
&=& \int_F\int_{Y_{n-1}} W\left( b\omega_ny_{n-1}\begin{pmatrix} 1 & a_1^{-1}x & \\
& 1 &
\\ &  & I_{n-2}
\end{pmatrix}\right)\psi(-y_{n-1})\psi(-x)|a_1|^{-1}dy_{n-1}dx   \\
&=& \int_F\int_{Y_{n-1}} W\left( b\omega_n \begin{pmatrix} a_1^{-1}
&  \\  & I_{n-1}
\end{pmatrix} y_{n-1}\begin{pmatrix} 1 & x & \\
& 1 &
\\ &  & I_{n-2}
\end{pmatrix}\begin{pmatrix} a_1
&  \\  & I_{n-1}
\end{pmatrix} \right)\psi(-y_{n-1})\psi(-x)|a_1|^{-(n-1)}dy_{n-1}dx
\end{eqnarray*}

Put $y_n=y_{n-1}\begin{pmatrix} 1 & x & \\  & 1 & \\ &  & I_{n-2}
\end{pmatrix}\in Y_n=N_n$, the above integral becomes

\begin{eqnarray*}
&=&\int_{Y_n} W\left( b\omega_n\begin{pmatrix} a_1^{-1} &  \\  &
I_{n-1}
\end{pmatrix}y_n\begin{pmatrix} a_1 &  \\  & I_{n-1}
\end{pmatrix}\right)\psi^{-1}(y_n)|a_1|^{-(n-1)}dy_n            \\
&=& j_{\pi}\left( b\omega_n\begin{pmatrix} a_1^{-1} &  \\  & I_{n-1}
\end{pmatrix}\right)W\begin{pmatrix} a_1 &  \\  & I_{n-1}
\end{pmatrix}|a_1|^{-(n-1)}               \\
\end{eqnarray*}
where the last equality follows form the identity

\[
\int_{Y_n} \left(\pi\begin{pmatrix} a_1 &  \\  & I_{n-1}
\end{pmatrix} .W\right)\left(b\omega_n\begin{pmatrix} a_1^{-1} &  \\  &
I_{n-1}
\end{pmatrix}y_n \right)\psi^{-1}(y_n)dy_n
=
\]
\[
j_{\pi}\left( b\omega_n\begin{pmatrix} a_1^{-1} &  \\  & I_{n-1}
\end{pmatrix}\right)\left(\pi\begin{pmatrix} a_1 &  \\  & I_{n-1}
\end{pmatrix} .W\right)\left( I\right)
\]
which is the definition of $j_{\pi}$. \\

Now we get

\begin{eqnarray*}
& &\int_{Y_{n-1}} W\left( b\omega_ny_{n-1}
\right)\psi(-y_{n-1})dy_{n-1} \\
&=&M_{n-1}(0)=\int_{F^\times}  \widehat{M}_{n-1}(a_1)da_1 \\
&=&\int_{F^\times} j_{\pi}\left( b\omega_n\begin{pmatrix} a_1^{-1} &  \\
& I_{n-1} \end{pmatrix}\right) W\begin{pmatrix} a_1 &  \\  & I_{n-1}
\end{pmatrix}|a_1|^{-(n-1)}da_1   \\
\end{eqnarray*}
which is exactly what we  want to show in (\ref{2}).     \\

Now set $h_2=\begin{pmatrix} 1 & \\ x_{21} & a_2 \\ & & I_{n-2}
\end{pmatrix}$ with $x_{21}\in F$, $a_2\in F^\times$, we also use
$h_2$ to denote the left upper corner $2\times 2$ matrix. We first
note the following identity of product of matrices

\[
b\omega_n h_2^{-1}= \begin{pmatrix} I_{n-2} & & \\ & 1 &
-b_{n-1}b_n^{-1}a_2^{-1}x_{21} \\ & & 1
\end{pmatrix}\begin{pmatrix} I_{n-2} & & \\ & a_2^{-1} & \\ & & 1
\end{pmatrix}b\omega_n
\]
Hence

\begin{eqnarray*}
W_v\left( b\omega_nh_2^{-1} y_{n-1}h_2 \right)&=&
W_v\left(\begin{pmatrix} I_{n-2} & & \\ & 1 &
-b_{n-1}b_n^{-1}a_2^{-1}x_{21} \\ & & 1
\end{pmatrix}\begin{pmatrix} I_{n-2} & & \\ & a_2^{-1} & \\ & & 1
\end{pmatrix}b\omega_ny_{n-1} h_2 \right) \\
&=&\psi(-b_{n-1}b_n^{-1}a_2^{-1}x_{21})
W_{\pi(h_2)v}\left(\begin{pmatrix} I_{n-2} & & \\ & a_2^{-1} & \\ &
& 1
\end{pmatrix}b\omega_ny_{n-1} \right)
\end{eqnarray*}

and

\begin{eqnarray*}
&& j_{\pi}\left( b\omega_nh_{2}^{-1} \begin{pmatrix} a_1^{-1} &  \\
& I_{n-1}
\end{pmatrix}\right)\\
&=& j_{\pi}\left(\begin{pmatrix} I_{n-2} & & \\ & 1 &
-b_{n-1}b_n^{-1}a_2^{-1}x_{21} \\ & & 1
\end{pmatrix}\begin{pmatrix} I_{n-2} & & \\ & a_2^{-1} & \\ & & 1
\end{pmatrix}b\omega_n \begin{pmatrix} a_1^{-1} &  \\  &
I_{n-1}
\end{pmatrix}  \right) \\
&=&\psi(-b_{n-1}b_n^{-1}a_2^{-1}x_{21})j_{\pi}\left(\begin{pmatrix}
I_{n-2} & & \\ & a_2^{-1} & \\ & & 1
\end{pmatrix}b\omega_n \begin{pmatrix} a_1^{-1} &  \\  &
I_{n-1}
\end{pmatrix} \right)
\end{eqnarray*}

Now apply (\ref{2}) to $b'=diag(b_1,...,b_{n-2}, a_2^{-1}b_{n-1},
b_n)$, and $W=W_{\pi(h_2)v}$, we will get

\[
\int_{Y_{n-1}} W_{\pi(h_2)v}\left( b'\omega_ny_{n-1}
\right)\psi(-y_{n-1})dy_{n-1} =
\]
\[
\int_{F^{\times}} j_{\pi}\left( b'\omega_n\begin{pmatrix} a_1^{-1} &  \\
& I_{n-1}
\end{pmatrix}\right) W_{\pi(h_2)v}\begin{pmatrix} a_1 &  \\  & I_{n-1}
\end{pmatrix}|a_1|^{-(n-1)}da_1
\]
Multiply by $\psi(-b_{n-1}b_n^{-1}a_2^{-1}x_{21})$ on both sides,
and then

\[
\int_{Y_{n-1}} W\left( b\omega_nh_2^{-1} y_{n-1}h_2
\right)\psi(-y_{n-1})dy_{n-1} =
\]

\begin{equation}
\int_{F^{\times}} j_{\pi}\left( b\omega_nh_{2}^{-1} \begin{pmatrix} a_1^{-1} &  \\
& I_{n-1}  \end{pmatrix}\right) W\left(\begin{pmatrix} a_1 &  \\  &
I_{n-1}  \end{pmatrix}h_2 \right) |a_1|^{-(n-1)}da_1
 \label{3}
 \end{equation}

Write $y_{n-1}=y_{n-2}\begin{pmatrix} I_2 & u_2 & \\ & 1 & \\ & &
I_{n-3} \end{pmatrix}$
 with $u_2$ a column vector in $F^2$, then the left side of (\ref{3}) is

\begin{equation}
\int_{Y_{n-1}} W\left( b\omega_n y_{n-2}\begin{pmatrix} I_2 &
h_2^{-1}.u_2 &
\\ & 1 & \\ & & I_{n-3} \end{pmatrix}
   \right)\psi(-y_{n-2})\psi\begin{pmatrix} I_2 & -u_2 & \\ & 1 & \\ & & I_{n-3} \end{pmatrix}
    |a_2|^{(n-3)} dy_{n-1} \label{4}
\end{equation}

Now put

\[
M_{n-2}(u_2)=\int_{Y_{n-2}} W\left( b\omega_n y_{n-2}\begin{pmatrix}
I_2 &  u_2 & \\ & 1 & \\ & & I_{n-3}
\end{pmatrix}
   \right)\psi(-y_{n-2}) dy_{n-2}
\]

By Lemma \ref{le2}, $M_{n-2}(u_2)$ is compactly supported, and its
Fourier inversion formula is

\[
M_{n-2}(u_2)=\int_{F^2} \widehat{M}_{n-2}(v_2)\psi(v_2
u_2)dv_2=\int_{F\times F^{\times}} \widehat{M}_{n-2}(x,y)\psi((x,y)
u_2)dxdy
\]
if we write $v_2=(x,y)$ as a row vector, with $x\in F$, $y\in
F^{\times}$, and the last equality follows from the facts that
$dxdy$ are additive Haar
measures and $F\times F^{\times}$ is of full measure in $F^2$. \\

Put $u_2=0$, we get

\[
M_{n-2}(0)=\int_{F\times F^{\times}} \widehat{M}_{n-2}(x,y)dxdy
\hspace{3cm} (*)
\]

Now we compute the Fourier coefficient $\widehat{M}_{n-2}(x,y)$ with
$x=x_{21}\in F$, $y=a_2\neq 0$, we have

\begin{eqnarray*}
 && \widehat{M}_{n-2}(x_{21},a_2) \\
 &=& \int_{F^2}\int_{Y_{n-2}}
W\left( b\omega_n y_{n-2}\begin{pmatrix} I_2 &  u_2 & \\ & 1 &
\\ & & I_{n-3}
\end{pmatrix}
   \right)\psi(-y_{n-2})\psi(-(x_{21},a_2) u_2)dy_{n-2}du_2   \\
&=& \int_{F^2}\int_{Y_{n-2}} W\left( b\omega_n
y_{n-2}\begin{pmatrix} I_2 &  u_2 & \\ & 1 &
\\ & & I_{n-3}
\end{pmatrix}
   \right)\psi(-y_{n-2})\psi\begin{pmatrix} I_2 & -h_2.u_2 & \\ & 1 & \\ & &
I_{n-3} \end{pmatrix}dy_{n-2}du_2 \\
&=& \int W\left( b\omega_n y_{n-2}\begin{pmatrix} I_2 & h_2^{-1}.u_2
&
\\ & 1 & \\ & & I_{n-3} \end{pmatrix}
   \right)\psi(-y_{n-2})\psi\begin{pmatrix} I_2 & -u_2 & \\ & 1 & \\ & & I_{n-3} \end{pmatrix}
   |a_2|^{-1} dy_{n-2}du_2
\end{eqnarray*}

Thus by (\ref{4}),

\[
|a_2|^{-(n-2)}\int_{Y_{n-1}} W\left( b\omega_nh_2^{-1} y_{n-1}h_2
\right)\psi(-y_{n-1})dy_{n-1} =\widehat{M}_{n-2}(x_{21},a_2)
\]

and then by (\ref{3}),

\[
|a_2|^{-(n-2)}\int_{F^{\times}} j_{\pi}\left( b\omega_nh_{2}^{-1} \begin{pmatrix} a_1^{-1} &  \\
& I_{n-1}  \end{pmatrix}\right) W\left(\begin{pmatrix} a_1 &  \\  &
I_{n-1}  \end{pmatrix}h_2 \right) |a_1|^{-(n-1)}da_1
=\widehat{M}_{n-2}(x_{21},a_2)
\]

Plug it into the Fourier inversion formula $(*)$, we find

\[
\int_{Y_{n-2}} W\left( b\omega_n y_{n-2} \right)\psi(-y_{n-2})
dy_{n-2}= M_{n-2}(0)=\int_{F\times F^{\times}}\int_{F^{\times}}
\]

\begin{equation}
  j_{\pi}\left(b\omega_n\begin{pmatrix} a_1 & \\
x_{21} & a_2 & \\ & & I_{n-2}
\end{pmatrix}^{-1} \right) W\left(b\omega_n\begin{pmatrix} a_1 & \\ x_{21} & a_2 & \\ & & I_{n-2}
\end{pmatrix}\right)|a_1|^{-(n-1)}da_1
|a_2|^{-(n-2)}dx_{21}da_2 \label{5}
\end{equation}
where we write

\[
\begin{pmatrix} a_1 &  \\  &
I_{n-1}  \end{pmatrix}h_2=\begin{pmatrix} a_1 & \\ x_{21} & a_2 & \\
& & I_{n-2}
\end{pmatrix}.
\]

Inductively, we will have

$$\int_{Y_2} W(b\omega_ny_2)\psi(-y_2)dy_2=$$
$$\int j_{\pi}\left(b\omega_n \begin{pmatrix} a_1 & \\ x_{21} & a_2 \\ & & \ddots
\\ x_{n-2,1} & \cdots & x_{n-2,n-3} & a_{n-2} \\ & & & &  I_2  \end{pmatrix}^{-1}
\right)W\begin{pmatrix} a_1 & \\ x_{21} & a_2 \\ & & \ddots \\
x_{n-2,1} & \cdots & x_{n-2,n-3} & a_{n-2}
 \\ & & & &  I_2  \end{pmatrix} $$
\begin{equation}
|a_1|^{-(n-1)}da_1|a_2|^{-(n-2)}dx_{21}da_2\cdots
|a_{n-2}|^{-2}dx_{n-2,1}\cdots dx_{n-2,n-3}da_{n-2} \label{6}
\end{equation}
where the right side is an iterated integral, $a_i$ is integrated
over $F^{\times}\subset F$ for $i=1,...,n-2$, $x_{ij}$ is integrated
over $F$ for all relevant $i,j$ here, and all measures are additive
self-dual Haar measures on $F$. \\

To prove the weak kernel formula, set
$$h_{n-1}=\begin{pmatrix} & I_{n-2} \\ x_{n-1,1} & \cdots & x_{n-1,n-2} & a_{n-1} \\ & & & & 1 \end{pmatrix}$$
where $x_{n-1,i}\in F$, $i=1,2,...,n-2$, $a_{n-1}\in F^\times$. We
also use $h_{n-1}$ to denote the left upper corner matrix of size
$(n-1)\times (n-1)$. Note that we have identity

\[
b\omega_n h_{n-1}^{-1}=\begin{pmatrix} 1 & \\ & 1 &
a_{n-1}^{-1}b_3^{-1}b_2x_{n-1,n-2} & ... &
a_{n-1}^{-1}b_n^{-1}b_2x_{n-1,1} \\ & & I_{n-2}
\end{pmatrix}\begin{pmatrix} 1 & \\ & a_{n-1}^{-1} & \\ & & I_{n-2}
\end{pmatrix}b\omega_n
\]
Then we carry out the same argument as we derived (\ref{3}) from
(\ref{2}), then by (\ref{6}), we will have

$$\int_{Y_2} W(b\omega_n h_{n-1}^{-1}y_2h_{n-1})\psi(-y_2)dy_2=\int $$
$$ j_{\pi}\left(b\omega_nh_{n-1}^{-1} \begin{pmatrix} a_1 & \\ x_{21} & a_2
\\ & & \ddots \\ x_{n-2,1} & \cdots & x_{n-2,n-3} & a_{n-2} \\ & & & &  I_2
\end{pmatrix}^{-1} \right)W\left( \begin{pmatrix} a_1 & \\ x_{21} & a_2
\\ & & \ddots \\ x_{n-2,1} & \cdots & x_{n-2,n-3} & a_{n-2} \\ & & & &  I_2
\end{pmatrix}h_{n-1}\right) $$
\begin{equation}
|a_1|^{-(n-1)}da_1|a_2|^{-(n-2)}dx_{21}da_2\cdots
|a_{n-2}|^{-2}dx_{n-2,1}\cdots dx_{n-2,n-3}da_{n-2} \label{7}
\end{equation}

The left side integral is

\begin{equation}
\int_{F^{n-1}} W \left( b\omega_n\begin{pmatrix} I_{n-1} &
h_{n-1}^{-1}.u_{n-1} \\ & 1
\end{pmatrix} \right)\psi\begin{pmatrix} I_{n-1} & -u_{n-1} \\ & 1
\end{pmatrix}du_{n-1} \hspace{1cm}
\label{8}
\end{equation}

Now let
$$M_1(u_{n-1})=W\left(b\omega_n \begin{pmatrix} I_{n-1} & u_{n-1} \\ & 1 \end{pmatrix}\right)$$
which is a compactly supported function in column vector $u_{n-1}$.
Its Fourier inversion formula is

\begin{eqnarray*}
M_1(u_{n-1}) &=&\int_{F^{n-1}}
\widehat{M}_1(v_{n-1})\psi(-v_{n-1}u_{n-1})dv_{n-1} \\
&=&\int_{F^{n-2}\times F^{\times}}
\widehat{M}_1(z_1,...,z_{n-1})\psi(-(z_1,...,z_{n-1})u_{n-1})dz_1...dz_{n-1}
\end{eqnarray*}
where we write $v_{n-1}=(z_1,...,z_{n-1})$, with $z_1,...,z_{n-2}\in
F$, $z_{n-1}\in F^{\times}$, and the last equality follows from the
facts that $dz_1...dz_{n-1}$ is the additive Haar measure and
$F^{n-2}\times F^{\times}$ is of full measure in $F^{n-1}$. \\

Put $u_{n-1}=0$, we get

\[
M_1(0)=\int_{F^{n-2}\times F^{\times}}
\widehat{M}_1(z_1,...,z_{n-1})dz_1...dz_{n-1} \hspace{2.5cm}  (**)
\]

We compute the Fourier coefficient $\widehat{M}_1(z_1,...,z_{n-1})$,
with $z_1=x_{n-1,1},...,z_{n-2}=x_{n-1,n-2}\in F, z_{n-1}=a_{n-1}\in
F^{\times}$,

\begin{eqnarray*}
& &\widehat{M}_1(x_{n-1,1},...,x_{n-1,n-2}, a_{n-1}) \\
&=&\int_{F^{n-1}} W\left(b\omega_n \begin{pmatrix} I_{n-1} & u_{n-1}
\\ & 1
\end{pmatrix}\right)\psi(-(x_{n-1,1},...,x_{n-1,n-2}, a_{n-1})u_{n-1})du_{n-1}   \\
&=& \int_{F^{n-1}} W\left(b\omega_n \begin{pmatrix} I_{n-1} &
u_{n-1}
\\ & 1
\end{pmatrix}\right) \psi\left(-\begin{pmatrix} I_{n-1} & h_{n-1}.u_{n-1} \\ & 1
\end{pmatrix}\right)du_{n-1} \\
&=& \int_{F^{n-1}} W\left(b\omega_n \begin{pmatrix} I_{n-1} &
h_{n-1}^{-1}.u_{n-1} \\ & 1
\end{pmatrix}\right)\psi\left(-\begin{pmatrix} I_{n-1} & u_{n-1} \\ & 1
\end{pmatrix}\right)|a_{n-1}^{-1}|du_{n-1}
\end{eqnarray*}

Thus by (\ref{8}), we have

\[
|a_{n-1}|^{-1}\int_{Y_2} W(b\omega_n
h_{n-1}^{-1}y_2h_{n-1})\psi(-y_2)dy_2=\widehat{M}_1(x_{n-1,1},...,x_{n-1,n-2},
a_{n-1})
\]

Thus by (\ref{7}),

\[
\widehat{M}_1(x_{n-1,1},...,x_{n-1,n-2},a_{n-1})= |a_{n-1}|^{-1}\int
\]
\[  j_{\pi}\left(b\omega_nh_{n-1}^{-1}
\begin{pmatrix} a_1 & \\ x_{21} & a_2
\\ & & \ddots \\ x_{n-2,1} & \cdots & x_{n-2,n-3} & a_{n-2} \\ & & & &  I_2
\end{pmatrix}^{-1} \right)
W\left( \begin{pmatrix} a_1 & \\ x_{21} & a_2
\\ & & \ddots \\ x_{n-2,1} & \cdots & x_{n-2,n-3} & a_{n-2} \\ & & & &  I_2
\end{pmatrix}h_{n-1}\right)
\]
\[
|a_1|^{-(n-1)}da_1|a_2|^{-(n-2)}dx_{21}da_2\cdots
|a_{n-2}|^{-2}dx_{n-2,1}\cdots dx_{n-2,n-3}da_{n-2}
\] \\

Plug it into $(**)$, we finally find

$$W(b\omega_n)=$$
$$\int j_{\pi}\left(b\omega_n \begin{pmatrix} a_1 & \\ x_{21} & a_2 \\ & & \ddots
\\ x_{n-1,1} & \cdots & x_{n-1,n-2} & a_{n-1} \\ & & & &  1  \end{pmatrix}^{-1}
\right)W\begin{pmatrix} a_1 & \\ x_{21} & a_2 \\ & & \ddots \\
x_{n-1,1} & \cdots & x_{n-1,n-2} & a_{n-1} \\ & & & &  1
\end{pmatrix} $$
$$|a_1|^{-(n-1)}da_1|a_2|^{-(n-2)}dx_{21}da_2\cdots |a_{n-1}|^{-1}dx_{n-1,1}
\cdots dx_{n-1,n-2}da_{n-1} $$ where the right side is an iterated
integral, $a_i$ is integrated over $F^{\times}\subset F$ for
$i=1,...,n-1$, $x_{ij}$ is integrated over $F$ for all relevant
$i,j$, and all measures are additive Haar measure on $F$. This
finishes the proof.

\end{proof}

$\bullet$ Since we don't have absolute convergence of right side
integral, we have to write it as an iterated integral. If we know
the local integrability of $j_{\pi}$, then using the same argument
as in Lemma 5.3, \cite{Baruch:2004}, one can show that the right
side integral is then absolutely convergent, and it equals

\[
\int_{N_{n-1}\backslash G_{n-1}} j_{\pi}\left(y\begin{pmatrix}
h^{-1} & \\ & 1\end{pmatrix} \right)W\begin{pmatrix} h & \\ &
1\end{pmatrix}dh
\]

$\bullet$ The space of functions $\left\{ W\begin{pmatrix} g & \\ &
1 \end{pmatrix}
 : W\in \mathcal{W} \right\}$ is the Kirillov model of $\pi$.
 Theorem \ref{thm1}
gives the action of the longest Weyl element $\omega_n$ on this
model in terms of Bessel functions. It thus follows that if we want
to show two supercuspidal representations are equivalent, it
suffices
to show they have the same Bessel functions. \\

$\bullet$ In order to generalize the above argument to generic
smooth irreducible representations of $GL(n,F)$, we need to know the
space $\mathcal{W}^0$, as defined in section 5 of
\cite{Baruch:2005}, is invariant under right translations by
elements like

\[
\begin{pmatrix} a_1 & \\ x_{21} & a_2 \\
& & \ddots \\ x_{n-1,1} & \cdots & x_{n-1,n-2} & a_{n-1}\\ & & & & 1
\end{pmatrix}.
\]
But this is not clear, and we plan to address this issue in
future. \\

\begin{corollary}
\label{cor1} Let $\widetilde{W}(g)=W(\omega_n\cdot {^t}g^{-1}) \in
\mathcal{W}(\widetilde{\pi}, \psi^{-1})$, then for any $b\omega_n$,
$b\in A_n$, we have

$$\widetilde{W}((b\omega_n)^{-1})=$$
$$\int j_{\pi}\left( \begin{pmatrix} a_1 & \\ x_{21} & a_2 \\ & & \ddots \\ x_{n-1,1}
& \cdots & x_{n-1,n-2} & a_{n-1} \\ & & & &  1  \end{pmatrix}
b\omega_n \right)\widetilde{W}
\begin{pmatrix} a_1 & \\ x_{21} & a_2 \\ & & \ddots \\ x_{n-1,1}
& \cdots & x_{n-1,n-2} & a_{n-1} \\ & & & &  1  \end{pmatrix} $$
$$|a_1|^{-(n-1)}da_1|a_2|^{-(n-2)}dx_{21}da_2\cdots |a_{n-1}|^{-1}dx_{n-1,1}
\cdots dx_{n-1,n-2}da_{n-1} $$ where the right side is an iterated
integral, $a_i$ is integrated over $F^{\times}\subset F$ for
$i=1,...,n-1$, $x_{ij}$ is integrated over $F$ for all relevant
$i,j$, and all measures are additive self-dual Haar measures on $F$.
\end{corollary}

\begin{proof}
Since $j_{\tilde{\pi},\psi^{-1}}(g)=j_{\pi,\psi}(g^{-1})$ for $g\in
B_n\omega_nB_n$, apply the above theorem.
\end{proof}

\section{Howe Vectors}

In this section, we will discuss Howe vectors, which were introduced
first by R.Howe. We will follow the exposition in \cite{Baruch:2005}
closely. Assume $\pi$ is irreducible and generic. \\

For a positive integer $m$, let $K_n^m=I_n+M_n(\mathfrak{p}^m)$,
here $\mathfrak{p}$ is the maximal ideal of $\mathcal{O}$. Use
$\varpi$ to denote an uniformizer of $F$. Let

$$d=\begin{pmatrix} 1 & \\ & \varpi^2 \\ & & \varpi^4 \\
& & & \ddots \\ & & & & \varpi^{2n-2}   \end{pmatrix}$$

Put $J_m=d^mK_n^md^{-m}$, $N_{n,m}=N_n\cap J_m$,
$\bar{N}_{n,m}=\bar{N}_n\cap J_m$, $\bar{B}_{n,m}=\bar{B}_n\cap
J_m$.
 Let $A_{n,m}=A_n\cap J_m$, then

$$J_m=\bar{N}_{n,m}A_{n,m}N_{n,m}= \bar{B}_{n,m}N_{n,m}$$

For $j\in J_m$, write $j=\bar{b}_jn_j$ with respect to the above
decomposition, as in \cite{Baruch:2005}, define a character $\psi_m$
on $J_m$ by

$$\psi_m(j)=\psi(n_j)$$

\begin{definition}
\label{def3} $W\in \mathcal{W}$ is called a Howe vector of $\pi$ if
for any $m$ large enough, we have

\begin{equation}
W(gj)=\psi_m(j)W(g) \hspace{1cm} \label{9}
\end{equation}
for all $g\in G_n$, $j\in J_m$.
\end{definition}

For each $W\in \mathcal{W}(\pi,\psi)$, let $M$ be a positive
constant such that $R(K_{n}^M).W=W$ where $R$ denotes the action of
right multiplication. For any $m>3M$, put

\[
W_m(g)=\int_{N_{n,m}}W(gu)\psi^{-1}(u)du
\]
then by Lemma 7.1 in \cite{Baruch:2005}, we have

\[
W_m(gj)=\psi_m(j)W_m(g), \forall j\in J_m, \ \ \forall g\in G_n
\]
This gives the existence of Howe vectors when $m$ is large enough.
The following lemma establishes its uniqueness in
Kirillov model. \\

\begin{thm}
\label{thm10} Assume $W\in \mathcal{W}$ satisfying (\ref{9}). Let
$h\in G_{n-1}$, if

$$W\begin{pmatrix} h & \\ & 1 \end{pmatrix}\neq 0$$
then $h\in N_{n-1}\bar{B}_{n-1,m}$. Moreover

$$W\begin{pmatrix} h & \\ & 1 \end{pmatrix}=\psi(u)W(I)$$
if $h=u\bar{b}$, with $u\in N_{n-1}$, $\bar{b}\in \bar{B}_{n-1,m}$.
\end{thm}

$\bullet$ Howe vectors were first introduced by R.Howe in an
unpublished paper (\cite{Ho}), in which he proved certain existence
and uniqueness properties of such vectors based on Gelfand-Kazhdan
method. We will below give an elementary proof of this theorem which
calculates the Howe vectors in Kirillov models. This result also
provides Howe vector as an candidate for the `unramified' vector
other than new vectors even in the `ramified' representations. \\

\begin{proof}
We will use an inductive argument. Write

$$h=\begin{pmatrix} h_{11} & \cdots & h_{1,n-1} \\ & \cdots &
 \\ h_{n-1,1} & \cdots & h_{n-1,n-1} \end{pmatrix}$$

Take $$u=\begin{pmatrix} I_{n-1} & u \\  & 1 \end{pmatrix}=
\begin{pmatrix} 1 & 0 & \cdots & u_1 \\  & 1 & 0 \cdots & u_2 \\
 & & \cdots & & \\  & &  1 & u_{n-1} \\  & & &  1   \end{pmatrix}\in J_m$$

We have

$$\psi(u_{n-1})W\begin{pmatrix} h & \\ & 1 \end{pmatrix}=
W\left(\begin{pmatrix} h & \\ & 1 \end{pmatrix}
\begin{pmatrix} I_{n-1} & u \\  & 1 \end{pmatrix}\right)$$
$$=W\begin{pmatrix} h & h.u \\  & 1 \end{pmatrix}=
W\left(\begin{pmatrix} I_{n-1} & h.u \\  & 1 \end{pmatrix}
\begin{pmatrix} h & \\  & 1 \end{pmatrix}\right)$$
$$=\psi(\sum_{i=1}^{n-1} h_{n-1,i}u_i)
W\begin{pmatrix} h & \\  & 1 \end{pmatrix}$$

Since $W\begin{pmatrix} h & \\  & 1 \end{pmatrix}\neq 0$, we get

$$\psi(\sum_{i=1}^{n-2}h_{n-1,i}u_i+(h_{n-1,n-1}-1)u_{n-1} )=1$$

Note that $u_i\in \mathfrak{p}^{(2i-2n+1)m}$, $i=1,2,...,n-1$, it
follows that $h_{n-1,i}\in \mathfrak{p}^{(2n-1-2i)m}$, $i=1,2,...,
n-2$, and $h_{n-1,n-1}\in 1+\mathfrak{p}^m$. So we may write

$$h=\begin{pmatrix} I_{n-2} & y \\  & 1 \end{pmatrix}\begin{pmatrix}
g & \\  x & a \end{pmatrix}=\begin{pmatrix} I_{n-2} & y \\  & 1
\end{pmatrix}\begin{pmatrix} g & \\  & 1
\end{pmatrix}\begin{pmatrix} I_{n-2} & \\ x & a \end{pmatrix}$$ with
$x=(h_{n-1,1},..., h_{n-1,n-2} )$, $a=h_{n-1,n-1}$,
$y=h_{n-1,n-1}^{-1}{^t}(h_{1,n-1},h_{2,n-1},...,h_{n-2,n-1} )$,
$g=h_{n-2}-y\cdot x$, where $y$ is a column vector, $x$ is a row
vector, and

$$h_{n-2}=\begin{pmatrix} h_{11} & \cdots & h_{1,n-2} \\
& \cdots &
\\ h_{n-2,1} & \cdots & h_{n-2,n-2} \end{pmatrix}$$

Since $j=\begin{pmatrix} I_{n-2} & \\ x & a \end{pmatrix}\in J_m$,
and by the assumption on $W$, we get

$$W\begin{pmatrix} h & \\  & 1
\end{pmatrix}=W\left(\begin{pmatrix} I_{n-2} & y \\  & 1 \\ & & 1
\end{pmatrix}\begin{pmatrix} g & \\  & 1 \\ & & 1
\end{pmatrix}\begin{pmatrix} I_{n-2} & \\ x & a \\ & & 1
\end{pmatrix}\right)$$
$$=\psi(y_{n-2})\psi_m(j)W\begin{pmatrix} g & \\  & 1 \\ & & 1 \end{pmatrix}$$

Note that $\psi_m(j)=1$ and it follows that

$$W\begin{pmatrix} g & \\  & 1 \\ & & 1 \end{pmatrix} \neq 0$$
and now we can argue inductively to get the result.
\end{proof}

$\bullet$ It follows from this lemma that $W(I)\ne 0$ for Howe
vectors, we will normalize it so that $W(I)=1$. \\

We apply the kernel formula to Howe vector to prove the following
result. \\

\begin{proposition}
\label{prop2} Assume $\pi$ is supercuspidal. Fix
$b=diag(b_1,...,b_n)\in A_n$, choose $m$ large
enough so that \\

(1). $R(A_{n,m}).j_{\pi}(b\omega_n)=j_{\pi}(b\omega_n)$, and
$L(A_{n,m}).j_{\pi}(b\omega_n)=j_{\pi}(b\omega_n)$; \\

(2). $\frac{b_{i-1}}{b_i}\in
\mathfrak{p}^{-3m},i=3,...,n$       \\
Then

\[
W(b\omega_n)=vol(\bar{B}_{n-1,m})j_\pi(b\omega_n)
\]
\end{proposition}

\begin{proof} We first note that although we don't know whether the weak kernel
formula is absolutely convergent, but when applying it to Howe
vectors, by Theorem \ref{thm10}, Howe vectors have nice compact
support modulo $N_{n-1}$ in the Kirillov model, hence in this case
the weak kernel formula is absolutely convergent. Write

\[
x=\begin{pmatrix} 1 & \\ x_{21} & 1 \\ & & \ddots \\
x_{n-1,1} & \cdots & x_{n-1,n-2} & 1 \\ & & & &  1
\end{pmatrix}
\]
and

\[
a=diag(a_1,a_2,...,a_{n-1},1)
\]
Apply the kernel formula to Howe vector $W_m$ and by Theorem
\ref{thm10}, we find

\[
W_m(b\omega_n)=\int_{\bar{B}_{n-1,m}}
j_\pi(b\omega_n(xa)^{-1})da_1dx_{21}da_2\cdots dx_{n-1,1}\cdots
dx_{n-1,n-2}da_{n-1}
\]
Note that $b\omega_na^{-1}x^{-1}a\omega_nb^{-1}$ is a upper
triangular unipotent matrix and

\[
\psi(b\omega_na^{-1}x^{-1}a\omega_nb^{-1})=
\psi(-x_{21}\frac{a_1}{a_2}\frac{b_{n-1}}{b_{n}}-...-x_{n-1,n-2}
\frac{a_{n-2}}{a_{n-1}}\frac{b_2}{b_3})
\]
then the above integral equals

\[
\int_{\bar{B}_{n-1,m}}
\psi(-x_{21}\frac{a_1}{a_2}\frac{b_{n-1}}{b_{n}}-...-x_{n-1,n-2}
\frac{a_{n-2}}{a_{n-1}}\frac{b_2}{b_3})j_\pi(b\omega_na^{-1})da_1dx_{21}da_2\cdots
dx_{n-1,1}\cdots dx_{n-1,n-2}da_{n-1}
\]
\[
=\int_{\bar{B}_{n-1,m}}j_\pi(b\omega_na^{-1})da_1dx_{21}da_2\cdots
dx_{n-1,1}\cdots dx_{n-1,n-2}da_{n-1}
\]
since $a_i\in A_{n,m}, i=1,...,n-1$, $\frac{b_{i-1}}{b_i}\in
\mathfrak{p}^{-3m},i=3,...,n$ by assumption (2), and $x_{i,i-1}\in
\mathfrak{p}^{3m},i=2,...,n-1$. \\

Now by assumption (1), $j_\pi(b\omega_na^{-1})=j_\pi(b\omega_n)$,
and eventually we have

\[
W(b\omega_n)=j_\pi(b\omega_n)\int_{\bar{B}_{n-1,m}}da_1dx_{21}da_2\cdots
dx_{n-1,1}\cdots
dx_{n-1,n-2}da_{n-1}=vol(\bar{B}_{n-1,m})j_\pi(b\omega_n)
\]
\end{proof}

\section{Bessel distributions and Bessel functions}

In this section, we will show for supercuspidal representation
$\pi$, the Bessel function $j_0(g)$ defined in section 3 via Bessel
distribution, is equal to the Bessel function $j_{\pi}(g)$, defined
in section 2 via uniqueness of Whittaker functional. We first review
some results and constructions in \cite{Baruch:2001}, which will be
useful for our purpose. \\

As in \cite{Baruch:2001}, use $L(\omega_n).f$ to denote the left
translation by $\omega_n$ on $f$, for $f$ a locally constant
compactly supported function of $G_n$. Then this action induces an
action on distributions, still denoted as $L(\omega_n)$. We now
consider the distribution $J=L(\omega_n).B_{l,l'}$, where $B_{l,l'}$
is the Bessel distribution defined in section 3. An important result
proved in \cite{Baruch:2001} is that, the restriction of $J$ to
$\bar{N}_nA_nN_n$  is given by the locally constant function $j_0$,
and the restriction of $B_{l,l'}$ to $N_n\omega_nA_nN_n$ is then
given by $j_{\pi}=L(\omega_n).j_0$. We next describe the method used
to prove
this fact in section 3.3 of \cite{Baruch:2001}. \\

We first transform the distribution $J$ on $Y=\bar{N}_nA_nN_n$ to a
distribution $\sigma_J$ on $A_n$ using the constructions in
\cite{Baruch:2001}. For every $f\in C_c^{\infty}(Y)$, define
$\beta_f\in C_c^{\infty}(A_n)$ by

\[
\beta_f(a)=\int_{\bar{N}_n\times
N_n}f(\bar{u}_1au_2)\psi(-\bar{u}_1)\psi(-u_2)du_1du_2
\]
where $\psi(\bar{u}_1)=\psi(\omega_n\bar{u}_1\omega_n)$. Then by
Proposition 1.12 in \cite{Shalika:1974}, the map sending $f$ to
$\beta_f$ is a surjective map from $C_c^{\infty}(Y)$ onto
$C_c^{\infty}(A_n)$, and there exists a unique distribution
$\sigma_J$ on $A_n$ with

\[
J(f)=\sigma_J(\beta_f)
\]
Moreover, if the distribution $\sigma_J$ on $A_n$ is given by a
locally constant function $\phi(a)$, then the distribution $J$ on
$Y$ is given by the locally constant function
$\psi(\bar{u}_1)\psi(u_2)\phi(a)\Delta^{-1}(a)$, where $\Delta$
satisfies $dg=\Delta(a)d\bar{u}_1dadu_2$ on $Y$. \\

To show $\sigma_J$ is given by some locally constant $\phi(a)$, we
need to introduce the following concept as in \cite{Baruch:2001},
specializing to our case.

\begin{definition}
\label{def4} Let $\Theta$ be a distribution on $A_n$. $\Theta$ is
said to be \textit{admissible} if for any $a\in A_n$, there exists
some compact open subgroup $K$ (depending on $a$) of $A_n$, such
that for every nontrivial character $\chi$ of $K$ we have
$\Theta(\chi_a)=0$, where $\chi_a$ is the function defined on $aK$
by $\chi_a(ak)=\chi(k)$, $k\in K$.
\end{definition}

We then have the following lemma.

\begin{lemma}
\label{le4} A distribution $\Theta$ on $A_n$ is admissible if and
only if there exists a locally constant function $\theta$ on $A_n$
such that
\[
\Theta(f)=\int_{A_n} \theta(a)f(a)da
\]
for all $f\in C_c^{\infty}(A_n)$. Moreover, the value of $\theta(a)$
if given by $\frac{1}{vol(K(a))}\Theta(1_a)$, where $K(a)$ is any
compact open subgroup of $A_n$ satisfying Definition \ref{def4}, and
$1$ denotes the trivial character of $K(a)$.
\end{lemma}

\begin{proof}
This is exactly Lemma 3.2 in \cite{Baruch:2001} applied to our case.
\end{proof}

In view of the above discussion, it suffices to show $\sigma_J$ is
admissible, and this is done in section 3.3 of \cite{Baruch:2001}.
Moreover fix $a=diag(a_1,...,a_n)\in A_n$, choose $m$ large enough.
More precisely, let $M>0$ be a positive constant as in Corollary 3.5
in \cite{Baruch:2001}, we then
require $m$ to satisfy that \\

(1). $\psi$ is trivial on $\omega_n a \bar{N}_{n,m} a^{-1}\omega_n$
and $\omega_n a^{-1} \bar{N}_{n,m} a\omega_n$; \\

(2). $R(A_{n,m}).j_{\pi}(a\omega_n)=j_{\pi}(a\omega_n)$, and
$L(A_{n,m}).j_{\pi}(a\omega_n)=j_{\pi}(a\omega_n)$. \\

(3). $m\ge M$. \\
As $\psi$ has conductor exactly $\CO$, if $m$ is large (1) can then
be satisfied. Because $j_\pi$ is locally constant, for a given $a\in
A_n$, (2) can be satisfied once $m$ is large. Hence one can choose
$m$ large enough satisfying all the above (1),(2),(3). \\

Then $A_{n,m}=A_n\cap J_m$, which is a compact open subgroup of
$A_n$, will satisfy Definition \ref{def4} by assumption (1) as shown
in \cite{Baruch:2001}, and then by Lemma \ref{le4},
$\phi(a)=\sigma_J(\frac{1}{vol(A_{n,m})}\chi_a)$ where $\chi_a$ is
the characteristic function of $aA_{n,m}$.   \\

Now let $\phi_1,\phi_2$ be a multiple of the characteristic function
of $\bar{N}_n\cap K_n$, and $N_n\cap K_n$, respectively, with

$$\int_{\bar{N}_n}
\phi_1(\bar{u})\psi(-\omega_n\bar{u}\omega_n)d\bar{u}=1$$ and
$$\int_{N_n} \phi_2(u)\psi(-u) du=1$$
then $\Phi_a=\frac{1}{vol(A_{n,m})} \phi_1\chi_a\phi_2$ is a locally
constant compactly supported function on $\bar{N}_nA_nN_n$. Note
that $\frac{1}{vol(A_{n,m})}\chi_a=\beta_{\Phi_a}$. \\

Hence $j_{\pi}(a)=\phi(a)\Delta^{-1}(a)
=\sigma_J(\frac{1}{vol(A_{n,m})}\chi_a)\Delta^{-1}(a)
=J(\Phi_a)\Delta^{-1}(a)$. Then

\[
j_0(\omega_na)=j_{\pi}(a)=J(\Phi_a)\Delta^{-1}(a)
=B_{l,l'}(L(\omega_n)\Phi_a)\Delta^{-1}(a).
\] \\

Note that $L(\omega_n).\Phi_a$ belongs to
$C_c^{\infty}(N_n\omega_nA_nN_n)$, and hence can be viewed as an
element in $C_c^{\infty}(G_n)$. Choose another positive integer
$m_1$ large enough so that \\

(1). $L(\bar{b}).L(\omega_n).\Phi_a=L(\omega_n).\Phi_a$ for any
$b\in \bar{B}_{n-1,m_1}$; \\

(2). $m_1\ge m$; \\

(3). $\frac{a_i}{a_{i-1}}\in \mathfrak{p}^{-3m_1}, i=3,...,n$. \\
Since $L(\omega_n).\Phi_a\in C_c^{\infty}(G_n)$, hence it is
bi-invariant under some open compact subgroup, then (1) is satisfied
if $m_1$ is large. Thus we can choose $m_1$ large enough to satisfy
all above (1),(2),(3).   \\

Apply Lemma \ref{le1} to $f=L(\omega_n)\Phi_a$,
$\hat{W}=\hat{W}_{m_1}$ the Howe vector, then we find by Theorem
\ref{thm10}, the right hand side of Lemma \ref{le1} is

\begin{eqnarray*}
&&\int_{N_{n-1}\backslash G_{n-1}} B_{l,l'}\left( L\begin{pmatrix} h
& \\ & 1 \end{pmatrix}.L(\omega_n).\Phi_a  \right)\hat{W}_{m_1} \begin{pmatrix} h & \\
& 1 \end{pmatrix}dh   \\
&=& B_{l,l'}(L(\omega_n)\Phi_a) vol(\bar{B}_{n-1,m_1})
\end{eqnarray*}
While the left side integral of Lemma \ref{le1} is
\begin{eqnarray*}
&& \int f(g)\hat{W}_{m_1}(g^{-1})dg \\
&=& \frac{1}{vol(A_{n,m})}\int \phi_1(\omega_nu_1\omega_n)\psi(-u_1)
\phi_2(u_2)\psi(-u_2)du_1du_2 \int_{A_{n,m}}
\hat{W}_{m_1}(h^{-1}a^{-1}\omega_n ) \Delta(ah) dh   \\
&=&\frac{1}{vol(A_{n,m})}\Delta(a)
\int_{A_{n,m}}\hat{W}_{m_1}(h^{-1}a^{-1}\omega_n ) \Delta(h) dh \\
\end{eqnarray*}
For any $h'\in A_{n,m_1}\subset A_{n,m}$, $h\in A_{n,m}$, since
$j_{\tilde{\pi}}(g^{-1})=j_{\pi}(g)$, we have

\[
j_{\tilde{\pi}}(h'h^{-1}a^{-1}\omega_n)=j_{\pi}(\omega_nahh'^{-1})
=j_{\pi}(\omega_nah)=j_{\tilde{\pi}}(h^{-1}a^{-1}\omega_n)
\]

\[
j_{\tilde{\pi}}(h^{-1}a^{-1}\omega_nh')=j_{\pi}(h'^{-1}\omega_nah)
=j_{\pi}(\omega_nah)=j_{\tilde{\pi}}(h^{-1}a^{-1}\omega_n)
\]
and $\frac{h_ia_i}{h_{i-1}a_{i-1}}\in \mathfrak{p}^{-3m_1},
i=3,...,n$ if $h=diag(h_1,...,h_n)$. Hence we can apply Proposition
\ref{prop2} to $\hat{W}_{m_1}(h^{-1}a^{-1}\omega_n )$ in the above
integral and get

\begin{eqnarray*}
&& \int f(g)\hat{W}_{m_1}(g^{-1})dg \\
&=& \frac{1}{vol(A_{n,m})}\Delta(a) \int_{A_{n,m}}
j_{\tilde{\pi}}(h^{-1}a^{-1}\omega_n)vol(\bar{B}_{n-1,m_1})\Delta(h)dh
\\
&=& \frac{1}{vol(A_{n,m})}\Delta(a)\int_{A_{n,m}}
j_{\pi}(\omega_nah)vol(\bar{B}_{n-1,m_1})dh
\end{eqnarray*}
where the last equality follows from the facts that
$j_{\tilde{\pi}}(g^{-1})=j_{\pi}(g)$ and $\Delta(h)=1$ when
restricted to $A_{n,m}$. \\

Now by assumption (2) on $m$, the above equals

\[
\Delta(a)j_{\pi}(\omega_na)vol(\bar{B}_{n-1,m_1})
\]
which is the left hand side of Lemma \ref{le1}. \\

Combining both sides of Lemma \ref{le1}, we get

\[
\Delta(a)j_{\pi}(\omega_na)vol(\bar{B}_{n-1,m_1})=B_{l,l'}(L(\omega_n)\Phi_a)
vol(\bar{B}_{n-1,m_1})
\]

Note that $B_{l,l'}(L(\omega_n)\Phi_a)=j_0(\omega_na)\Delta(a)$,
immediately we have

\begin{thm}
\label{thm2} Assume $\pi$ is supercuspidal. For all $g\in
N_nA_n\omega_n N_n$, we have

 $$j_0(g)=j_{\pi}(g)$$
 \end{thm}

\section{General Case}

The aim of this section is to generalize \textbf{Theorem} \ref{thm2}
from supercuspidal case to general generic case. Now assume $\pi$ is
an irreducible admissible smooth generic representation of $G_n$
with Whittaker model $\CW(\pi,\psi)$. The point is to prove a weak
kernel formula for Howe vectors of $\pi$. This weak kernel formula
 is expected to hold for a wider class of
Whittaker functions, but currently we are only able to prove it for
Howe
vectors, which is sufficient for our purpose. \\

The proof is essentially the same as the proof of \textbf{Theorem}
\ref{thm1} with necessary modifications. Use $W_m$ to denote the
normalized Howe vector of level $m$ of $\pi$ as in section 5.  \\

\begin{thm}
\label{thm3} (weak kernel formula) For any $b\omega_n$,
$b=diag(b_1,...,b_n)\in A_n$, if $m$ is large enough, we have

$$W_m(b\omega_n)=$$
$$\int j_{\pi}\left(b\omega_n \begin{pmatrix} a_1 & \\ x_{21} & a_2 \\ & & \ddots \\
x_{n-1,1} & \cdots & x_{n-1,n-2} & a_{n-1} \\ & & & &  1
\end{pmatrix}^{-1} \right) W_m\begin{pmatrix} a_1 & \\ x_{21} & a_2 \\
& & \ddots \\ x_{n-1,1} & \cdots & x_{n-1,n-2} & a_{n-1}\\ & & & & 1
\end{pmatrix} $$
$$|a_1|^{-(n-1)}da_1|a_2|^{-(n-2)}dx_{21}da_2\cdots |a_{n-1}|^{-1}dx_{n-1,1}\cdots dx_{n-1,n-2}da_{n-1} $$
where the right side is an iterated integral, $a_i$ is integrated
over $F^{\times}\subset F$ for $i=1,...,n-1$, $x_{ij}$ is integrated
over $F$ for all relevant $i,j$, and all measures are additive
self-dual Haar measures on $F$.
\end{thm}

We first note $W_m\begin{pmatrix} a_1 & \\ x_{21} & a_2 \\
& & \ddots \\ x_{n-1,1} & \cdots & x_{n-1,n-2} & a_{n-1}\\ & & & & 1
\end{pmatrix}\neq 0$ if and only if $\begin{pmatrix} a_1 & \\ x_{21} & a_2 \\
& & \ddots \\ x_{n-1,1} & \cdots & x_{n-1,n-2} & a_{n-1}\\ & & & & 1
\end{pmatrix}$ $\in \bar{B}_{n,m}$ by \textbf{Theorem} 5.2, in which case $W_m\begin{pmatrix} a_1 & \\ x_{21} & a_2 \\
& & \ddots \\ x_{n-1,1} & \cdots & x_{n-1,n-2} & a_{n-1}\\ & & & & 1
\end{pmatrix}=1$. This is very important for our proof in this
special case. \\

Introduce notations $h_i=\begin{pmatrix} &I_{i-1}&&& \\
x_{i,1}&...&x_{i,i-1}&a_i& \\ &&&& I_{n-i}\end{pmatrix}$. We also
use $h_i$ to denote the left upper corner $i\times i$ matrix when
there is no confusion. Note that
$\begin{pmatrix} a_1 & \\ x_{21} & a_2 \\
& & \ddots \\ x_{n-1,1} & \cdots & x_{n-1,n-2} & a_{n-1}\\ & & & & 1
\end{pmatrix}$ $=h_1h_2...h_{n-1}$. \\

\begin{proof}
Now let
$$M_1(u_{n-1})=W_m\left(b\omega_n \begin{pmatrix} I_{n-1} & u_{n-1} \\ & 1 \end{pmatrix}\right)$$
which is a compactly supported function in column vector $u_{n-1}$
by by Theorem 5.7 and Theorem 7.3 in \cite{Baruch:2005} as $W_m\in
\CW^0$ if $m$ is large enough. Its Fourier inversion formula is

\begin{eqnarray*}
M_1(u_{n-1}) &=&\int_{F^{n-1}}
\widehat{M}_1(v_{n-1})\psi(-v_{n-1}u_{n-1})dv_{n-1} \\
&=&\int_{F^{n-2}\times F^{\times}}
\widehat{M}_1(z_1,...,z_{n-1})\psi(-(z_1,...,z_{n-1})u_{n-1})dz_1...dz_{n-1}
\end{eqnarray*}
where we write $v_{n-1}=(z_1,...,z_{n-1})$, with $z_1,...,z_{n-2}\in
F$, $z_{n-1}\in F^{\times}$, and the last equality follows from the
facts that $dz_1...dz_{n-1}$ is the additive Haar measure and
$F^{n-2}\times F^{\times}$ is of full measure in $F^{n-1}$. \\

Put $u_{n-1}=0$, we get

\begin{equation}
M_1(0)=\int_{F^{n-2}\times F^{\times}}
\widehat{M}_1(z_1,...,z_{n-1})dz_1...dz_{n-1} \hspace{2.5cm}
\label{10}
\end{equation}

By the same computations as in section 4, we find

\[
|a_{n-1}|^{-1}\int_{Y_2} W(b\omega_n
h_{n-1}^{-1}y_2h_{n-1})\psi(-y_2)dy_2=\widehat{M}_1(x_{n-1,1},...,x_{n-1,n-2},
a_{n-1})
\]

So $(\ref{10})$ becomes

\[
W_m(b\omega_n)=\int_{F^{n-2}\times F^{\times}}
\widehat{M}_1(x_{n-1,1},...,x_{n-1,n-2},a_{n-1})dx_{n-1,1}...dx_{n-1,n-2}da_{n-1}
\]
\begin{equation}
=\int_{F^{n-2}\times F^{\times}} |a_{n-1}|^{-1}\int_{Y_2}
W_m(b\omega_n
h_{n-1}^{-1}y_2h_{n-1})\psi(-y_2)dy_2dx_{n-1,1}...dx_{n-1,n-2}da_{n-1}
\ \ \label{11}
\end{equation} \\

\textbf{Claim 1}: As a function of $h_{n-1}$, $\int_{Y_2}
W_m(b\omega_n h_{n-1}^{-1}y_2h_{n-1})\psi(-y_2)dy_2$ has support
in $\bar{B}_{n,m}$. \\

\textit{Proof of Claim 1:} The proof is similar to \textbf{Theorem}
5.2. Take

\[
u=\begin{pmatrix} I_{n-1} & u \\  & 1 \end{pmatrix}=
\begin{pmatrix} 1 & 0 & \cdots & u_1 \\  & 1 & 0 \cdots & u_2 \\
 & & \cdots & & \\  & &  1 & u_{n-1} \\  & & &  1   \end{pmatrix}\in J_m
\]
then
\[
\psi(u_{n-1})\int_{Y_2} W_m(b\omega_n
h_{n-1}^{-1}y_2h_{n-1})\psi(-y_2)dy_2=\int_{Y_2} W_m(b\omega_n
h_{n-1}^{-1}y_2h_{n-1}u)\psi(-y_2)dy_2
\]
\[
=\int_{Y_2} W_m(b\omega_n
h_{n-1}^{-1}y_2h_{n-1}uh_{n-1}^{-1}h_{n-1})\psi(-y_2)dy_2=\int_{Y_2}
W_m(b\omega_n
h_{n-1}^{-1}y_2h_{n-1})\psi(-y_2)\psi(h_{n-1}uh_{n-1}^{-1})dy_2
\]
where in the last equality we change variable
$y_2h_{n-1}uh_{n-1}^{-1}\to y_2$. \\

Thus we find that if $\int_{Y_2} W_m(b\omega_n
h_{n-1}^{-1}y_2h_{n-1})\psi(-y_2)dy_2\neq 0$, then

\[
\psi(u_{n-1})=\psi(h_{n-1}uh_{n-1}^{-1})
\]
As $u\in J_m$ is arbitrary, this forces $h_{n-1}\in
\bar{B}_{n,m}$, which proves the claim. \\

\hspace{14cm} $\Box$ \\

Let's continue the proof of the theorem. Compare $(\ref{11})$ with
the desired formula in theorem and note the support of $W_m$ and the
claim, it suffices to show

\[
\int_{Y_2} W_m(b\omega_n h_{n-1}^{-1}y_2h_{n-1})\psi(-y_2)dy_2=\int
\]
\[  j_{\pi}\left(b\omega_nh_{n-1}^{-1}
\begin{pmatrix} a_1 & \\ x_{21} & a_2
\\ & & \ddots \\ x_{n-2,1} & \cdots & x_{n-2,n-3} & a_{n-2} \\ & & & &  I_2
\end{pmatrix}^{-1} \right)
W_m\left( \begin{pmatrix} a_1 & \\ x_{21} & a_2
\\ & & \ddots \\ x_{n-2,1} & \cdots & x_{n-2,n-3} & a_{n-2} \\ & & & &  I_2
\end{pmatrix}h_{n-1}\right)
\]
\[
|a_1|^{-(n-1)}da_1|a_2|^{-(n-2)}dx_{21}da_2\cdots
|a_{n-2}|^{-2}dx_{n-2,1}\cdots dx_{n-2,n-3}da_{n-2}
\]
for $h_{n-1}\in \bar{B}_{n,m}$. \\

By properties of Howe vectors, this is equivalent to
\[
\int_{Y_2} W_m(b\omega_n h_{n-1}^{-1}y_2)\psi(-y_2)dy_2=\int
\]
\[  j_{\pi}\left(b\omega_nh_{n-1}^{-1}
\begin{pmatrix} a_1 & \\ x_{21} & a_2
\\ & & \ddots \\ x_{n-2,1} & \cdots & x_{n-2,n-3} & a_{n-2} \\ & & & &  I_2
\end{pmatrix}^{-1} \right)
W_m \begin{pmatrix} a_1 & \\ x_{21} & a_2
\\ & & \ddots \\ x_{n-2,1} & \cdots & x_{n-2,n-3} & a_{n-2} \\ & & & &  I_2
\end{pmatrix}
\]
\begin{equation}
|a_1|^{-(n-1)}da_1|a_2|^{-(n-2)}dx_{21}da_2\cdots
|a_{n-2}|^{-2}dx_{n-2,1}\cdots dx_{n-2,n-3}da_{n-2} \ \ \ \  \ \
\label{12}
\end{equation}
for $h_{n-1}\in \bar{B}_{n,m}$. \\

To prove $(\ref{12})$, let
\[
M_2(u_{n-2})=\int_{Y_2} W_m \left(b\omega_n h_{n-1}^{-1}y_2
\begin{pmatrix} I_{n-2}& u_{n-2}& \\ & 1& \\ &&1 \end{pmatrix} \right)\psi(-y_2)dy_2
\]
which is compactly supported function in column vector $u_{n-2}$ by
Theorem 5.7 and 7.3 in \cite{Baruch:2005}. Its Fourier inversion
formula is

\begin{eqnarray*}
M_2(u_{n-2}) &=&\int_{F^{n-2}}
\widehat{M}_2(v_{n-2})\psi(-v_{n-2}u_{n-2})dv_{n-2} \\
&=&\int_{F^{n-3}\times F^{\times}}
\widehat{M}_2(z_1,...,z_{n-2})\psi(-(z_1,...,z_{n-2})u_{n-2})dz_1...dz_{n-2}
\end{eqnarray*}
where we write $v_{n-2}=(z_1,...,z_{n-2})$, with $z_1,...,z_{n-3}\in
F$, $z_{n-2}\in F^{\times}$, and the last equality follows from the
facts that $dz_1...dz_{n-2}$ is the additive Haar measure and
$F^{n-3}\times F^{\times}$ is of full measure in $F^{n-2}$. \\

Put $u_{n-2}=0$, we get

\begin{equation}
M_2(0)=\int_{F^{n-2}\times F^{\times}}
\widehat{M}_2(z_1,...,z_{n-2})dz_1...dz_{n-2} \hspace{2.5cm}
\label{13}
\end{equation}

Similar computations as in section 4 shows that
\[
\widehat{M}_2(x_{n-2,1},...,x_{n-2,n-3},a_{n-2})=|deth_{n-2}|^{-2}
\]
\[
\int_{F^{n-2}} \int_{Y_2} W_m\left(
b\omega_nh_{n-1}^{-1}h_{n-2}^{-1}y_2
\begin{pmatrix} I_{n-2}&u_{n-2}& \\ & 1 & \\ &&1  \end{pmatrix}h_{n-2}
\right)\psi(-y_2)\psi\left( -\begin{pmatrix} I_{n-2}&u_{n-2}& \\ & 1
& \\ &&1  \end{pmatrix} \right) dy_2 du_{n-2}
\]
\[
=|deth_{n-2}|^{-2}\int_{Y_3} W_m(
b\omega_nh_{n-1}^{-1}h_{n-2}^{-1}y_3 h_{n-2} )\psi(-y_3)dy_3
\]

So $(\ref{13})$ becomes
\[
\int_{Y_2} W_m \left(b\omega_n h_{n-1}^{-1}y_2
 \right)\psi(-y_2)dy_2=
 \]
 \begin{equation}
\int_{F^{n-2}\times F^{\times}} |deth_{n-2}|^{-2} \int_{Y_3} W_m(
b\omega_nh_{n-1}^{-1}h_{n-2}^{-1}y_3 h_{n-2}
)\psi(-y_3)dy_3dx_{n-2,1}...da_{n-2} \ \ \ \  \ \ \label{14}
\end{equation} \\

\textbf{Claim 2}: As a function of $h_{n-2}$, $\int_{Y_3} W_m(
b\omega_nh_{n-1}^{-1}h_{n-2}^{-1}y_3 h_{n-2} )\psi(-y_3)dy_3$ has
support in $\bar{B}_{n,m}$. \\

\textit{Proof of the Claim 2}: Take
\[
u=\begin{pmatrix} I_{n-2}&u_{n-2}& \\ &1& \\ &&1   \end{pmatrix}\in
J_m
\]
then argue completely the same as the proof of Claim 1. The details
will be omitted. \\

\hspace{14cm} $\Box$ \\

Compare $(\ref{14})$ with $(\ref{12})$ and note the support of $W_m$
and Claim 2, then it suffices to show that
\[
\int_{Y_3} W_m( b\omega_nh_{n-1}^{-1}h_{n-2}^{-1}y_3 h_{n-2}
)\psi(-y_3)dy_3= \int
\]
\[  j_{\pi}\left(b\omega_nh_{n-1}^{-1}h_{n-2}^{-1}
\begin{pmatrix} a_1 & \\ x_{21} & a_2
\\ & & \ddots \\ x_{n-3,1} & \cdots & x_{n-3,n-4} & a_{n-3} \\ & & & &
I_3
\end{pmatrix}^{-1} \right)
W_m \left(\begin{pmatrix} a_1 & \\ x_{21} & a_2
\\ & & \ddots \\ x_{n-3,1} & \cdots & x_{n-3,n-4} & a_{n-3} \\ & & & &
I_3
\end{pmatrix}h_{n-2}\right)
\]
\[
|a_1|^{-(n-1)}da_1|a_2|^{-(n-2)}dx_{21}da_2\cdots dx_{n-2,1}\cdots
dx_{n-3,n-4}da_{n-3}
\]
for $h_{n-2}\in \bar{B}_{n,m}$, which is equivalent to
\[
\int_{Y_3} W_m( b\omega_nh_{n-1}^{-1}h_{n-2}^{-1}y_3
)\psi(-y_3)dy_3= \int
\]
\[  j_{\pi}\left(b\omega_nh_{n-1}^{-1}h_{n-2}^{-1}
\begin{pmatrix} a_1 & \\ x_{21} & a_2
\\ & & \ddots \\ x_{n-3,1} & \cdots & x_{n-3,n-4} & a_{n-3} \\ & & & &
I_3
\end{pmatrix}^{-1} \right)
W_m \begin{pmatrix} a_1 & \\ x_{21} & a_2
\\ & & \ddots \\ x_{n-3,1} & \cdots & x_{n-3,n-4} & a_{n-3} \\ & & & &
I_3
\end{pmatrix}
\]
\begin{equation}
|a_1|^{-(n-1)}da_1|a_2|^{-(n-2)}dx_{21}da_2\cdots dx_{n-2,1}\cdots
dx_{n-3,n-4}da_{n-3} \ \ \ \  \ \ \label{15}
\end{equation} \\

To prove $(\ref{15})$, inductively, it suffices to show that
\[
\int_{Y_{n-1}} W_m\left( b\omega_nh_{n-1}^{-1}...h_2^{-1} y_{n-1}
\right)\psi(-y_{n-1})dy_{n-1} =
\]
\[
\int_{F^{\times}} j_{\pi}\left( b\omega_n h_{n-1}^{-1}...h_2^{-1}\begin{pmatrix} a_1^{-1} &  \\
& I_{n-1}
\end{pmatrix}\right) W_m\begin{pmatrix} a_1 &  \\  & I_{n-1}
\end{pmatrix}|a_1|^{-(n-1)}da_1
\]
which can be proved completely in the same way as the proof of
\textbf{Theorem} $\ref{thm1}$. Thus the proof of the theorem is
finished. \\

\end{proof}

Now use the same method as in section 6 , together with the above
weak kernel formula for Howe vectors $W_m$ with $m$ large enough, we
can show that for irreducible generic admissible representation
$\pi$ of $G_n$, the Bessel functions $j_{\pi}$ defined via
uniqueness of Whittaker models, coincide with the Bessel function
$j_0$ defined via Bessel distributions, which generalize Theorem
$\ref{thm2}$ to general generic representations. As the proof is
completely the same as Theorem $\ref{thm2}$ we omit the details, and
simply state the result as follows. \\

\begin{thm}
\label{thm4} If $\pi$ is an irreducible admissible smooth generic
representation of $G_n$, then for any $g\in N_n\omega_nA_nN_n$, we
have
\[
j_{\pi}(g)=j_0(g)
\]

\end{thm}

\vspace{1cm}

\end{document}